%% file: tangnorKorngenLip-v04.tex
\documentclass[a4paper]{amsart}

\def\mylabelonoff{off}
\def\allowdisbrkyesno{yes}
\def\numberingtheoremsectionyesno{no}
\def\numberingequationsectionyesno{no}
\def\pagesizeextendednormal{extended}

\def\reportudemathyesno{no}
\def\reportudemathnumber{SM-UDE-791}
\def\reportudemathyear{2016}
\def\reportudematheingang{\mydate}

\def\mytitle{On Korn's First Inequality for Mixed Tangential and Normal Boundary Conditions\\
on Bounded Lipschitz Domains in $\reals^N$}
\def\mytitlerepude{On Korn's First Inequality\\ 
for Mixed Tangential and Normal Boundary Conditions\\
on Bounded Lipschitz-Domains in $\reals^N$}
\def\myshorttitle{On Korn's First Inequality}
\def\myauthorone{Sebastian Bauer}
\def\myauthortwo{Dirk Pauly}
\def\myauthors{\myauthorone\quad\&\quad\myauthortwo}
\def\myaddressone{Fakult\"at f\"ur Mathematik,
Universit\"at Duisburg-Essen, Campus Essen, Germany}

\def\myemailone{sebastian.bauer.seuberlich@uni-due.de}
\def\myemailtwo{dirk.pauly@uni-due.de}
\def\mykeywords{Korn inequality, tangential and normal boundary conditions, Boltzmann equation}
\def\mysubjclass{49J40 / 82C40 / 76P05}
\def\mydate{\today}

\usepackage{ifthen,esint}
\input header-pauly

\title[\sc\myshorttitle]{\Large\sf\mytitle}
\author{\myauthorone}
\author{\myauthortwo}
\address{\myaddressone}
\email[\myauthorone]{\myemailone}
\email[\myauthortwo]{\myemailtwo}
\keywords{\mykeywords}
\subjclass{\mysubjclass}
\date{\mydate}

\newcommand{\so}{\mathfrak{so}}
\newcommand{\soN}{\so(N)}
\newcommand{\rmR}{\mathcal{R}}

\newcommand{\K}{\mathcal{K}}
\newcommand{\calL}{\mathcal{L}}
\newcommand{\sfS}{\mathsf{S}}
\newcommand{\Sz}{\sfS_{0}}
\newcommand{\Som}{\sfS(\om)}
\newcommand{\Szom}{\Sz(\om)}

\begin{document}


\ifthenelse{\equal{\reportudemathyesno}{yes}}
{\preprintudemath{\mytitlerepude}{\myauthors}{\reportudemathnumber}{\reportudemathyear}{\reportudematheingang}}
{}


\begin{abstract}
We prove that for bounded Lipschitz domains in $\rN$
Korn's first inequality holds for vector fields satisfying
homogeneous mixed tangential and normal boundary conditions.
\end{abstract}

\maketitle
\tableofcontents


\section{Introduction}

Recently, motivated by \cite{desvillettesvillanikornnormal,desvillettesvillanitrendglequiboltzmann}
and inspired by the ideas and techniques presented 
in \cite{paulymaxconst0,paulymaxconst1,paulymaxconst2} for estimating the Maxwell constants,
we have shown in \cite{bauerpaulytangnormkorn}
that Korn's first inequality, i.e.,
\begin{align}
\mylabel{firstkornsqrttwo}
\normltom{\na v}
&\leq\ck\,\normltom{\sym\na v},
\end{align}
holds with $\ck=\sqrt2$ for all vector fields $v\in\hoom$ satisfying (possibly mixed)
homogeneous normal or homogenous tangential boundary conditions
and for all piecewise $\coo$-domains $\om\subset\rN$, $N\geq2$,
with concave boundary parts. In this contribution, we extend \eqref{firstkornsqrttwo}
to any bounded (strong) Lipschitz domain $\om\subset\rN$, $N\geq2$. 
As pointed out in \cite{desvillettesvillanitrendglequiboltzmann},
this Korn inequality has an important application 
in statistical physics, more precisely in the study of relaxation to equilibrium of rarefied gases 
modeled by Boltzmann's equation.

\section{Preliminaries}

We will utilize the notations from \cite{bauerpaulytangnormkorn}.
Throughout this paper and unless otherwise explicitly stated, 
let $\om\subset\rN$, $N\geq2$, be a bounded domain
with strong Lipschitz boundary $\ga:=\p\om$, i.e.,
locally $\ga$ can be represented as a graph of a Lipschitz function.
As in \cite{bauerpaulytangnormkorn},
we introduce the standard scalar valued Lebesgue and Sobolev spaces 
by $\ltom$ and $\hoom$ as well as
$$\hocom:=\ovl{\cicom}^{\hoom},$$
respectively, where $\cicom$ denotes the test functions
yielding the usual Sobolev space $\hocom$ with zero boundary traces.
These definitions extend component-wise to vector or matrix, or more general tensor fields
and we will use the same notations for these spaces. 
Moreover, we will consistently denote functions by $u$ and vector fields by $v$.
We define the vector valued $\ho$-Sobolev space $\hoctom$ resp. $\hocnom$ 
as closure in $\hoom$ of the set of test vector fields 
\begin{align}
\mylabel{CtCn_def}
\cictom
:=\setb{v|_{\om}}{v\in\cicrN,\,\tcomp{v}=0},\qquad
\cicnom
:=\setb{v|_{\om}}{v\in\cicrN,\,\ncomp{v}=0},
\end{align}
respectively, generalizing homogeneous tangential resp. normal boundary conditions. 
Here, $\nu$ denotes the a.e. defined outer unit normal at $\ga$ giving a.e. 
the normal resp. tangential component
$$\ncomp{v}:=\nu\cdot v|_{\ga},\quad
\tcomp{v}:=v|_{\ga}-\ncomp{v}\nu$$
of $v$ on $\ga$.
We assume additionally that $\ga$ is decomposed into two relatively open subsets
$\gat$ and $\gan:=\ga\setminus\ovl{\gat}$ and introduce
the vector valued $\ho$-Sobolev space of mixed boundary conditions $\hoctnom$
as closure in $\hoom$ of the set of test vector fields
\begin{align}
\mylabel{Ctn_def}
\cictnom
:=\setb{v|_{\om}}{v\in\cicrN,\,\tcomp{v}|_{\gat}=0,\,\ncomp{v}|_{\gan}=0}.
\end{align}

\subsection{Korn's Second Inequality}

It is well known that Korn's second inequality 
can easily be proved by a simple $\hmo$-argument using Ne\u{c}as inequality.
Let us illustrate a simple and short proof:
In the sense of distributions we have e.g. for all vector fields $v\in\ltom$
that the components of $\na\na v$\footnote{We denote by $\na v$
the transpose of the Jacobian of $v$ and by $\na\na v$
the tensor of second derivatives of $v$.} 
consist only of components of $\na\sym\na v$, i.e.,
\begin{align}
\mylabel{nasymna}
\forall\,i,j,k=1,\dots,N\qquad
\p_{i}\p_{j}v_{k}
&=\p_{i}\sym_{j,k}\na v
+\p_{j}\sym_{i,k}\na v
-\p_{k}\sym_{i,j}\na v,
\end{align}
where $\sym_{j,k}T:=(\sym T)_{j,k}$.
By e.g. \cite[1.1.3 Lemma]{sohrbook} we have (for scalar functions)
the Ne\u{c}as estimate
\begin{align}
\mylabel{necas}
\exists\,c&>0&
\forall\,u&\in\ltom&
c\,\normltom{u}
&\leq\normhmoom{\na u}+\normhmoom{u}
\leq(\sqrt{N}+1)\normltom{u},
\end{align}
where $\hmoom:=\big(\hocom\big)'$ and e.g. by using the full $\hoom$-norm
$$\normhmoom{u}
:=\sup_{0\neq\varphi\in\hocom}
\frac{\scpltom{u}{\varphi}}{\norm{\varphi}_{\hoom}},\qquad
\normhmoom{\na u}
:=\sup_{0\neq\phi\in\hocom}
\frac{\scpltom{u}{\div\phi}}{\norm{\phi}_{\hoom}}.$$
For the original results of \eqref{necas} see the works of Ne\u{c}as, e.g.
\cite{necasbook1967,necaspaper1967a}, from the 1960s.

\begin{rem}
\mylabel{lbb}
Ne\u{c}as' estimate \eqref{necas} can be refined to 
\begin{align}
\mylabel{necasproj}
\exists\,c&>0&
\forall\,u&\in\ltzom:=\set{u\in\ltom}{\scpltom{u}{1}=0}&
c\,\normltom{u}
&\leq\normhmoom{\na u}
\leq\sqrt{N}\normltom{u}.
\end{align}
The best constant $c>0$ in \eqref{necasproj} is
also called inf-sup- or LBB-constant as by using the $\hoom$-half norm
\begin{align*}
c=\inf_{0\neq u\in\ltzom}
\frac{\normhmoom{\na u}}{\normltom{u}}
=\inf_{0\neq u\in\ltzom}\sup_{v\in\hocom}
\frac{\scpltom{u}{\div v}}{\normltom{u}\normltom{\na v}}
=c_{\mathsf{LBB}}.
\end{align*}
We note that the LBB-constant can be bounded from below by the 
inverse of the continuity constant $c_{A}$ of the $\ho$-potential operator
(often called Bogovskii operator)
$A:\ltzom\to\hocom$ with $\div Au=u$, i.e.,
\begin{align*}
\forall\,u&\in\ltzom&
\normltom{\na Au}
&\leq c_{A}\normltom{u}.
\end{align*}
This follows directly by setting $v:=Au$
(note that $\na Au\neq0$ for $0\neq u\in\ltzom$) and
$$c_{\mathsf{LBB}}
\geq\inf_{0\neq u\in\ltzom}
\frac{\normltom{u}^2}{\normltom{u}\normltom{\na Au}}
\geq\frac{1}{c_{A}}.$$
\end{rem}

We immediately get:

\begin{theo}
[Korn's second inequality]
\mylabel{seckorn}
There exists $c>0$ such that for all $v\in\hoom$
\begin{align*}
\normltom{\na v}
&\leq c\big(\normltom{\sym\na v}
+\normltom{v}\big).
\end{align*}
\end{theo}

\begin{proof}
Let $v\in\hoom$. Combining \eqref{nasymna} and \eqref{necas} we estimate
\begin{align*}
\normltom{\na v}
&\leq c\,\big(\normhmoom{\na\na v}+\normhmoom{\na v}\big)\\
&\leq c\,\big(\normhmoom{\na\sym\na v}+\normhmoom{\na v}\big)
\leq c\,\big(\normltom{\sym\na v}+\normltom{v}\big),
\end{align*}
showing the stated result.
\end{proof}

By standard mollification we see that the restrictions 
of $\cicrN$-vector fields to $\om$ are dense in 
$$\Som:=\set{v\in\ltom}{\sym\na v\in\ltom},$$
even if $\om$ just has the segment property. 
Especially $\hoom$ is dense in $\Som$. 
This shows immediately:

\begin{theo}
[$\ho$-regularity]
\mylabel{horeg}
It holds $\Som=\hoom$.
\end{theo}

\begin{proof}
Let $v\in\Som$. By density, there exists a sequence $(v_{n})\subset\hoom$
converging to $v$ in $\Som$. By Theorem \ref{seckorn} $(v_{n})$ is a Cauchy sequence
in $\hoom$ converging to $v$, yielding $v\in\hoom$.
\end{proof}

\begin{rem}
The latter arguments show, that for any domain allowing for Ne\u{c}as' estimate \eqref{necas}
Korn's second inequality Theorem \ref{seckorn} holds.
In these domains we have also the $\ho$-regularity Theorem \ref{horeg},
provided that the segment property holds.
\end{rem}

\begin{rem}
\label{lqremone}
\eqref{necas} is well known to hold also in the 
$\mathsf{L}^q/\,\mathsf{W}^{-1,q}$-setting for $1<q<\infty$.
As \eqref{nasymna} and the mollification techniques are available for general $q$,
it follows that Theorem \ref{seckorn} and Theorem \ref{horeg} immediately extend to the 
$\mathsf{L}^q/\,\mathsf{W}^{1,q}/\,\mathsf{S}^{q}$-setting 
for all $1<q<\infty$.
\end{rem}

\subsection{Poincar\'e Inequality for Elasticity}

To apply standard solution theories for linear elasticity, 
such as Fredholm's alternative for bounded domains
or Eidus' limiting absorption principle \cite{eiduslabp}
for exterior domains, it is most important to ensure for bounded domains the compact embedding 
\begin{align}
\mylabel{ecp}
\Som\hookrightarrow\ltom.
\end{align}
As long as Korn's second inequality, i.e., 
the continuous embedding $\Som\hookrightarrow\hoom$,
holds true, the compact embedding \eqref{ecp} follows immediately by Rellich's selection theorem, i.e.,
the compact embedding $\hoom\hookrightarrow\ltom$.
As shown in \cite{weckelacomp},
there are bounded irregular domains, more precisely bounded domains 
with the $p$-cusp property (H\"older boundaries), 
see \cite[Definition 3]{witschremmax} or \cite[Definition 2]{weckelacomp},
with $1<p<2$, for which Korn's second inequality fails
and so the embedding $\Som\subset\hoom$ 
by the closed graph theorem\footnote{The
identity mapping $\id_{\mathsf{S}}:\Som\to\hoom$ is continuous,
if and only if $\id_{\mathsf{S}}$ is closed,
if and only if $\Som\subset\hoom$.},
but the important compact embedding \eqref{ecp} remains valid.
More precisely, by \cite[Theorem 2]{weckelacomp} the compact embedding \eqref{ecp} 
holds for bounded domains having the $p$-cusp property with $1\leq p<2$\footnote{For $p=1$ 
the $1$-cusp property equals the strict cone property,
which itself holds for strong Lipschitz domains.},
and \eqref{ecp} implies immediately a Poincar\'e type inequality
for elasticity by a standard indirect argument. For this we define
$$\Szom
:=\set{v\in\Som}{\sym\na v=0}
=\set{v\in\ltom}{\sym\na v=0}.$$
It is well known that even for any domain $\om$
$$\Szom
=\rmR$$
holds, where $\rmR:=\set{Sx+a}{S\in\so\,\wedge\,a\in\rN}$ is the space rigid motions
and $\so=\soN$ the vector space of constant skew-symmetric matrices.
This follows easily for $v\in\Szom$ by approximating $\om$ by smooth domains $\om_{n}$,
in each of which $v_{n}:=v|_{\om_{n}}$ equals the same rigid motion $r\in\rmR$.

\begin{theo}
[Poincar\'e inequality for elasticity]
\mylabel{poincareela}
Let $\om$ be bounded and possess the $p$-cusp property with some $1\leq p<2$.
Then there exists $c>0$ such that for all $v\in\Som\cap\rmR^{\bot}$
$$\normltom{v}
\leq c\,\normltom{\sym\na v}.$$
Equivalently, for all $v\in\Som$
$$\normltom{v-r_{v}}
\leq c\,\normltom{\sym\na v},\qquad
r_{v}:=\pi_{\rmR}v.$$
\end{theo}

Here and throughout the paper, we denote orthogonality in $\ltom$ by $\bot$.
Moreover, $\pi_{\rmR}$ denotes the $\ltom$-orthogonal projector
onto the rigid motions $\rmR$. 

\begin{proof}
If the assertion was wrong, there exists a sequence $(v_{n})\subset\Som\cap\rmR^{\bot}$
with $\normltom{v_{n}}=1$ and $\normltom{\sym\na v_{n}}\to0$.
By \eqref{ecp} we can assume without loss of generality 
that $(v_{n})$ converges in $\ltom$ to some $v\in\ltom$.
But then $v\in\Szom\cap\rmR^{\bot}=\{0\}$,
in contradiction to $1=\normltom{v_{n}}\to\normltom{v}=0$.
\end{proof}

Under the assumptions of Theorem \ref{poincareela},
the variational static linear elasticity problem,
for $f\in\ltom$ find $v\in\Som\cap\rmR^{\bot}$ such that 
$$\forall\,\phi\in\Som\cap\rmR^{\bot}\qquad
\scpltom{\sym\na v}{\sym\na\phi}
=\scpltom{f}{\phi},$$
is uniquely solvable with continuous resp. compact inverse 
$\ltom\to\Som$ resp. $\ltom\to\ltom$, which shows that 
Fredholm's alternative holds for the corresponding reduced operators.

\section{Korn's First Inequality}

By Rellich's selection theorem, Theorem \ref{seckorn} 
and an indirect argument we can easily prove:

\begin{theo}
[Korn's first inequality without boundary conditions]
\mylabel{firstkorn}
There exists $c>0$ such that for all $v\in\hoom$ with $\na v\bot\,\so$
\begin{align}
\mylabel{firstkornbot}
\normltom{\na v}
&\leq c\,\normltom{\sym\na v}.
\intertext{Equivalently for all $v\in\hoom$}
\nonumber
\normltom{\na v-S_{v}}
&\leq c\,\normltom{\sym\na v},\qquad
S_{v}:=\frac{1}{|\om|}\skw\int_{\om}\na v.
\end{align}
\end{theo}

Here, $S_{v}=\pi_{\so}\na v$ is the $\ltom$-orthogonal projection 
of $\na v$ onto $\so$. 

\begin{proof}
The equivalence is clear by the orthogonal projection.\footnote{We can also 
compute it by hand: For $v\in\hoom$ with $\na v\bot\,\so$ we see
$$|S_{v}|^2
=\frac{1}{|\om|}\scp{\skw\int_{\om}\na v}{S_{v}}
=\frac{1}{|\om|}\scpltom{\na v}{S_{v}}
=0$$
since $S_{v}\in\so$. For $v\in\hoom$ and $T\in\so$ we have
\begin{align*}
\scpltom{\na v-S_{v}}{T}
&=\int_{\om}\scp{\skw\na v}{T}-\scpltom{S_{v}}{T}
=\scp{\int_{\om}\skw\na v}{T}-|\om|\scp{S_{v}}{T}
=0,
\end{align*}
implying $v+s_{v}\in\hoom$ with $\na(v+s_{v})=(\na v-S_{v})\bot\,\so$ 
and $\sym\na(v+s_{v})=\sym(\na v-S_{v})=\sym\na v$,
where $s_{v}(x):=S_{v}x$.}
If \eqref{firstkornbot} was wrong, there exists a sequence $(v_{n})\subset\hoom$
with $\na v_{n}\bot\,\so$ and $\normltom{\na v_{n}}=1$ and $\normltom{\sym\na v_{n}}\to0$.
Without loss of generality we can assume $v_{n}\bot\,\rN$.
By Poincare's inequality $(v_{n})$ is bounded in $\hoom$.
Thus, by Rellich's selection theorem we can assume without loss of generality 
that $(v_{n})$ converges in $\ltom$ to some $v\in\ltom$.
By Theorem \ref{seckorn} $(v_{n})$ is a Cauchy sequence in $\hoom$.
Therefore $(v_{n})$ converges in $\hoom$ to $v\in\hoom\cap(\rN)^{\bot}$
with $\sym\na v=0$ and $\na v\bot\,\so$. But then $\na v$ is even constant
and belongs to $\so$. 
Hence $\na v=0$\footnote{We note that even $v\in\rN$ holds and thus $v=0$.} 
in contradiction to $1=\normltom{\na v_{n}}\to\normltom{\na v}=0$.
\end{proof}

Using Poincare's inequality we immediately obtain:

\begin{cor}
[Korn's first inequality without boundary conditions]
\mylabel{firstkorncor}
There exists $c>0$ such that for all $v\in\hoom\cap(\rN)^{\bot}$ with $\na v\bot\,\so$
\begin{align*}
\norm{v}_{\hoom}
&\leq c\,\normltom{\sym\na v}.
\end{align*}
\end{cor}

In order to prove Korn's first inequality in $\hoctnom$ 
we  need a Poincar\'e type estimate on this space. 
It should be noted that in general mixed boundary conditions are not sufficient   
to rule out a kernel of the gradient operator. 
For example, consider the cube $\om:=(0, \,1)^3\subset \reals^3$ 
with $\gat$ being the union of the top and bottom
together with the constant vector field $r(x):=(0,\,0,\,1)^t$.
Then $r\in \hoctnom$. On this account, 
such constant vector fields have to be excluded separately.
 
\begin{lem}
[Poincar\'e inequality with tangential or normal boundary conditions]
\mylabel{poincaretannor} 
There exists $c>0$ such that for all 
$v\in\hoctnom\cap \big(\hoctnom\cap\reals^N\big)^\perp$
$$\normltom{v}
\leq c\,\normltom{\na v}.$$
\end{lem}

\begin{proof}
If the assertion was wrong, there exists some sequence 
$(v_{n})\subset\hoctnom\cap \big(\hoctnom\cap\reals^N\big)^\perp$
with $\normltom{v_{n}}=1$ and $\normltom{\na v_{n}}\to0$.
Thus, by Rellich's selection theorem we can assume without loss of generality 
that $(v_{n})$ converges in $\ltom$ to some $v\in\ltom$.
Hence, $(v_{n})$ is a Cauchy sequence in $\hoom$ 
and converges in $\hoom$ to $v\in\hoctnom\cap \big(\hoctnom\cap\reals^N\big)^\perp$
with $\na v=0$. Therefore, $v$ is a constant in $\hoctnom\cap\rN$ 
and must vanish in contradiction to $1=\normltom{v_{n}}\to\normltom{v}=0$.
\end{proof}

As an easy consequence we get

\begin{cor}
\mylabel{nahoclosed}
$\na\hoctnom$  is a closed subspace of $\ltom$.
\end{cor}

\begin{proof}
Let $(v_{n})\subset\hoctnom$ such that $\na v_{n}\to G\in\ltom$ in $\ltom$.
Without loss of generality we can assume 
$(v_{n})\subset\hoctnom\cap \big(\hoctnom\cap\reals^N\big)^\perp$,
otherwise we replace $v_n$ by 
$$\tilde{v}_n:=v_n-\pi_{\hoctnom\cap\rN}v_n\in\hoctnom\cap \big(\hoctnom\cap\reals^N\big)^\perp,$$
where $\pi_{\hoctnom\cap\rN}$ is the orthogonal projector onto $\hoctnom\cap\rN$.
Because of Lemma \ref{poincaretannor} $(v_{n})$ is a Cauchy sequence in $\hoctnom$,
which converges in $\hoom$ to $v\in\hoctnom$.
Hence, $G\ot\na v_{n}\to\na v\in\na\hoctnom$.
\end{proof}

To exclude the kernel of the $\sym\na$-operator on $\hoctnom$, we define
$$\K:=\set{\na v}{v\in\hoctnom,\,\sym\na v=0}
=\na\big(\rmR\cap\hoctnom\big)
=\so\cap\na\hoctnom.$$

\begin{theo}
[Korn's first inequality with tangential or normal boundary conditions]
\mylabel{firstkorntannor}
There exists $c>0$ such that for all $v\in\hoctnom$ with $\na v\bot\,\K$
\begin{align}
\mylabel{firstkorntannorineq}
\normltom{\na v}
\leq c\,\normltom{\sym\na v}.
\end{align}
Equivalently, for all $v\in\hoctnom$
\begin{align*}
\normltom{\na v-\pi_{\K}\na v}
&\leq c\,\normltom{\sym\na v}.
\end{align*}
\end{theo}

Here, $\pi_{\K}$ denotes the $\ltom$-orthogonal projector onto $\K$. 

\begin{proof}
Equivalence is again clear by the orthogonal projection.
If \eqref{firstkorntannorineq} was wrong, there exists a sequence $(v_{n})\subset\hoctnom$
with $\na v_{n}\bot\,\K$ and $\normltom{\na v_{n}}=1$ and $\normltom{\sym\na v_{n}}\to0$.
Without loss of generality we can assume $(v_{n})\subset\hoctnom\cap \big(\hoctnom\cap\reals^N\big)^\perp$.
By Lemma \ref{poincaretannor} $(v_{n})$ is bounded in $\hoom$,
and thus, using Rellich's selection theorem, we can assume without loss of generality 
that $(v_{n})$ converges in $\ltom$ to some $v\in\ltom$.
By Theorem \ref{seckorn} $(v_{n})$ is a Cauchy sequence in $\hoom$.
Therefore, $(v_{n})$ converges in $\hoom$ to $v\in\hoctnom$
with $\sym\na v=0$ and $\na v\bot\,\K$. 
But then, $\na v$ is even a constant in $\so$, i.e., $\na v\in\K$,
in contradiction to $1=\normltom{\na v_{n}}\to\normltom{\na v}=0$.
\end{proof}

\begin{rem}
\label{lqremtwo}
Similar to Remark \ref{lqremone}, all the results from Theorem \ref{firstkorn}
to Theorem \ref{firstkorntannor} extend to the 
$\mathsf{L}^q/\,\mathsf{W}^{1,q}$-setting for all $1<q<\infty$
with the obvious modifications. The same holds true for all results
presented in the subsequent sections.
\end{rem}

\subsection{Discussing the Set $\K$}

In this section we shall discuss which combinations of domains $\om$ and boundary parts $\gat$ 
allow for a non-constant rigid motion $r\in\hoctnom\cap\rmR$, i.e., 
$\K\neq\{0\}$. We start with the case $\gat=\ga$, i.e, 
with the full tangential boundary condition.

\begin{theo} 
If $\gat=\ga$, then $\K=\{0\}$ and there exists a constant $c>0$ 
such that for all $v\in\hoctom$
\begin{align*}	
\normltom{\na v}
&\leq c\,\normlt{\sym\na v}.
\end{align*}
\end{theo}

\begin{proof}
We give a proof by contradiction. 
Assume $r\in\rmR\cap\hoctom$ and $r\not=0$. 
Let us define the null space $\calN_r:=\setlr{x\in\rN}{r(x)=0}$.
Then $\calN_r$ is an empty set or an affine plane in $\rN$ 
with dimension $d_{\calN_r}\leq N-2$.
We recall that $\nu$ is the outer unit normal at $\ga$ defined a.e. on $\ga$ 
w.r.t. the $(N-1)$-dimensional Lebesgue measure. 
Since $r$ is normal on $\ga$, we conclude for almost all $x\in\ga\setminus\calN_r$ 
\begin{align}
\label{rnu}
\nu(x)=\pm\frac{r(x)}{\norm{r(x)}}.
\end{align}
Because $\om$ is locally on one side of the boundary $\ga$, 
the unit normal $\nu$ cannot change sign in \eqref{rnu}
in any connected component of $\ga\setminus\calN_r$. 
But since $d_{\calN_r}\leq N-2$, it follows that $\ga\setminus\calN_r$ is connected, 
and w.l.o.g.
\begin{align}
\label{sign}
\nu(x)=\frac{r(x)}{\norm{r(x)}}\qquad\text{for almost all }x\in\ga\setminus\calN_r.
\end{align}
As $\ga\cap\calN_r$ has measure zero, we can replace 
$\ga\setminus\calN_r$ by $\ga$ in \eqref{sign}. 
With Gau{\ss}' theorem we conclude
\begin{align*}
0=\int_\om\div r
=\int_\ga\nu\cdot r
=\int_{\ga}\norm{r}>0,
\end{align*}
a contradiction.
\end{proof}

Next we turn to the full normal boundary condition, i.e. $\gat=\emptyset$. 
In \cite{desvillettesvillanikornnormal}
it is proved that for smooth bounded domains $\om\subset\rN$ 
Korn's first inequality holds for all $v\in \hocnom$, i.e. $\K=\left\{0\right\}$,
if and only if $\om$ is not axisymmetric. 
Furthermore  an explicit upper bound 
on the constant is given.\footnote{In \cite{desvillettesvillanikornnormal}  
a $\co$-boundary is assumed, but it seems that for the proof of 
\cite[Lemma 4]{desvillettesvillanikornnormal} actually a $\ct$-boundary is 
needed in order to guaranty $\ho$-regularity of  
$\na\phi$, where $\phi$ is the solution of \cite[(14)]{desvillettesvillanikornnormal}.}
In that contribution and here axisymmetry is defined as follows.

\begin{defi}
\label{axisymmetry-def}
$\om$ is called axisymmetric if there
is a non-trivial rigid motion $r\in\calR$
tangential to the boundary $\ga$ of $\om$, 
i.e. $0\neq r\in\hocnom$.
\end{defi}

In a more elementary and canonical approach in $\rt$
a domain is called axisymmetric
w.r.t. an axis $a$ if it is a body of rotation around this axis.
In order to show that in $\rt$ both concepts coincide for bounded Lipschitz domains,
we make use of the invariance of a Lipschitz boundary 
under the flow of a tangential vector field.

\begin{pro}
\label{tangentialflow-lem}
\mylabel{normalboundaryconditioninvariantflow}
Let $\om\subset\rN$ be a (not necessarily bounded) 
domain with a (strong) Lipschitz boundary $\ga$ and 
$r:\rN\to\rN$ a locally Lipschitz continuous 
vector field that is tangential on $\ga$ a.e.~w.r.t. 
the $(N-1)$-dimensional Lebesgue measure on $\ga$.
Let $p\in\ga$ and let $t\mapsto\gamma(t)$ the maximal solution 
of the ordinary differential equation
\begin{equation}
\label{ODE}
\dot{\gamma}=r(\gamma),\quad 
\gamma(0)=p
\end{equation}
existing on the interval $I_p$.
Then for all $t\in I_p$
\begin{equation}
\gamma(t)\in\ga.
\end{equation} 
\end{pro}

This proposition is a variant of Nagumo's invariance theorem, 
see \cite[Theorem 2, p. 180]{aubincellina:84}, c.f. also \cite{nagumo:42},
where the tangential condition on $r$ is defined in terms of
a so called 'Bouligand contingent cone'.
As we need this statement for a Lipschitz boundary we give a
self-contained proof in the Appendix.

The next lemma states that for bounded domains in $\rt$ 
both definitions of axisymmetry coincide. 
An elementary proof is provided in the appendix.

\begin{lem}
\label{equivalence}
Let $\om\subset\rt$ be a bounded Lipschitz domain. 
\begin{itemize}
\item[\bf(i)] 
Assume $\sigma,b\in\rt$, $\norm{\sigma}=1$ 
and let $g=\setlr{\lambda\sigma+b}{\lambda\in\reals}$. 
Assume that $\om$ is axisymmetric w.r.t.~the axis $g$. 
Then the vector field $r$ with $r(x):=\sigma\wedge(x-b)$
is a rigid motion, which is tangential at $\ga$, 
i.e. $r\in\calR\cap\hocnom$.
\item[\bf(ii)]
Let $r\in\calR\cap\hocnom$, $r(x)=\omega\,\sigma\wedge x+b$ for all $x\in\rt$
with $\sigma,b\in\rt$, $\norm{\sigma}=1$ and $\omega\in\reals$.
Then $\omega\neq 0$, $\scp{b}{\sigma}=0$, and 
$\om$ is axisymmetric w.r.t.~the axis 
$g=\setlr{\lambda \sigma+\frac{1}{w}\sigma\wedge b}{\lambda\in\reals}$. 
\end{itemize}
\end{lem}

\begin{rem}
\label{unbounded} 
There are rigid motions tangential to the boundary 
of some \underline{unbounded} domains in  $\rt$,
which do not exhibit any axis of symmetry. 
Consider, for example,  
a domain $\om$ built from a plane square which simultaneously 
is lifted along and rotated around
the axis perpendicular to it, e.g.
$$\om:=\setb{\big(x_1\cos(t)-x_2\sin(t),x_1\sin(t)+x_2\cos(t),t\big)^t}
{\norm{x_1}+\norm{x_2}<1,\,t\in\reals}.$$
Then $r(x):=(-x_2, x_1, 1)^t$ is tangential to $\ga$. 
\end{rem}

Using Definition \ref{axisymmetry-def}, Korn's first inequality 
for normal boundary conditions is more or less obvious.  

\begin{theo}
Let $\gat=\emptyset$. 
Then Korn's first inequality holds for all $v\in\hocnom$,
if and only if $\K=\{0\}$,
if and only if $\om$ is not axisymmetric.
\end{theo}

\begin{proof}
The first 'if and only if' is just the assertion of Theorem \ref{firstkorntannor}. 
For the second 'if and only if' according to the definition of axisymmetry
the only remaining issue is to prove that there is no constant vector field
tangential to a bounded Lipschitz domain 
(in that case we would have a non-trivial rigid motion, 
which gives no contribution to $\K$). 
Assume that a constant vector $0\neq a\in\rN$ tangential to $\ga$ exists,
i.e. $a\in\hocnom$, and let $\hat{x}\in\ga$. 
Then according to Proposition \ref{tangentialflow-lem}
the unbounded curve $t\mapsto\hat{x}+ta$ would remain in $\ga$,
which contradicts the boundedness of $\om$.
\end{proof}

\begin{rem}
The latter proof shows that a bounded domain 
is axisymmetric if and only if 
there is a non-constant rigid motion tangential to the boundary.
\end{rem}

For mixed boundary conditions there are domains 
of rather special type with $\K\neq\{0\}$.
Consider, for example, a half cylinder 
$$\om:=\setlr{x\in\rt}{x_1>0,\,x_1^2+x_2^2<1,\,0<x_3<1},$$ 
or more generally, the domain
$$\om:=\setlr{(r\cos\phi,r\sin\phi,x_3)^t}{\phi_1<\phi<\phi_2,\,0<x_3<1,\,0< r<h(x_3)}$$
with $\gat:=\ga\cap\setlr{(r\cos \phi_{1/2},r\sin\phi_{1/2},x_3)^t}{0\leq r,\,0<x_3<1}$ 
and for some positive Lipschitz function $ h:\reals\to\reals$ and some $-\pi<\phi_1<\phi_2<\pi$.
Define $r(x):=(-x_2, x_1, 0)^t$. 
Then $r$ is a rigid motion and $r\in\hoctnom$.
In the next theorem we will show that in $\rt$ all bounded domains $\om$ 
with $\K\neq \left\{0\right\}$ are compositions of subdomains of this kind.

\begin{theo}
\label{gemischteRandbedingungen-lem}
Let $\om\subset\rt$ be a bounded Lipschitz domain 
and let $\emptyset\neq\gat\neq\ga$.
Assume that there is a non-constant rigid motion $r\in\calR\cap\hoctnom$, 
$r(x)=\omega\,\sigma\wedge x+b$ for all $x\in\rt$ 
with $\omega\in\reals$ and $\norm{\sigma}=1$. 
Define $g_r\subset\rt$ by
$g_r:=\setlr{\lambda\,\sigma+\frac{1}{\omega}\,\sigma\wedge b}{\lambda\in\reals}$.
Then $\scp{\sigma}{b}=0$, $\gat$ is a subset of 
a union of affine planes, where each of these planes contains $g_r$. 
Every connected component of $\gan$ 
is a subset of a surface which is axisymmetric w.r.t.~$g_r$.
\end{theo}

By this theorem the aforementioned cube, i.e.
$\om=(0,1)^3\subset\rt$ with
$\gat$ being the union of the top and bottom faces,
has a trivial kernel $\K=\{0\}$, 
which means Korn's first inequality Theorem \ref{firstkorntannor} 
holds on $\hoctnom$, 
while Poincar\'e's inequality Lemma \ref{poincaretannor} only holds
on $\hoctnom\cap\big((0,0,1)^t\big)^\perp$.

\begin{proof} 
First we note that the scalar-product $\scp{\sigma}{b}$ is independent 
of the chosen Cartesian coordinates, i.e. if we choose
another positively oriented Euclidian coordinate system $(y_1, y_2,y_3)$ and represent the 
vector field $r$ by means of the $y$-coordinates, 
then there exist vectors $\sigma_y,b_y\in\rt$ 
with $\norm{\sigma_y}=1$ and 	
$r(y)=\omega\,\sigma_y\wedge y+b_y$ for all $y\in\rt$. 
Furthermore $\scp{\sigma_y}{b_y}=\scp{\sigma}{b}$.
In the same way the representation of the axis $g_r$ associated 
to $r$ is independent of the Cartesian coordinates chosen;
in $y$-coordinates we have 
$g_r=\setlr{\lambda\,\sigma_y+\frac{1}{\omega}\sigma_y\wedge b_y}{\lambda\in\reals}$.

Suppose $r\in\calR\cap\hoctnom$ and that $r$ is not constant. 
We fix some $p\in\gat$ 
together with a neighborhood $U\subset\rt$ of $p$, 
an open subset $V\subset\reals^2$, 
Euclidian coordinates $(x_1,x_2,x_3)=(x',x_3)$ 
and a Lipschitz map $h:V\subset\reals^2\to\reals$,
such that for all $x\in U$ we have $x=(x',x^3)\in\gat$ 
if and only if $x^3=h(x')$.
Since $r$ is normal and by Rademacher's theorem, we have
\begin{align}
\label{normal-equation}
r(x',h(x'))
&=f(x')\big(\na_{x'}h(x'),-1\big)^t
\end{align}
with some function $f:V\subset\reals^2\to\reals$ a.e.~in $V$.
	
In $x$-coordinates $r$ can be represented by $r(x)=\omega\,\sigma\wedge x+b$ 
with some $b,\sigma\in\rt$, $\norm{\sigma}=1$ and $0\neq\omega\in\reals$.
From \eqref{normal-equation} we conclude
\begin{align}	
\label{K1}
b_1+\omega\,\sigma_2h(x')-\omega\,\sigma_3x_2 
&=f(x')\p_1h(x'),\\
\label{K2}
b_2+\omega\,\sigma_3x_1-\omega\,\sigma_1h(x')
&=f(x')\p_2h(x'),\\
\label{K3}
b_3+\omega\,\sigma_1x_2-\omega\,\sigma_2x_1
&=-f(x').
\end{align}
We differentiate (in the sense of distributions) \eqref{K1} 
w.r.t.~$x_2$ and \eqref{K2} w.r.t.~$x_1$, 
compute the  difference as well as the sum of the resulting equations, 
and conclude using \eqref{K3}
\begin{align}
\label{sigma3}
\sigma_3
&=\sigma_1\p_1h+\sigma_2\p_2h,\\ 
\label{gemischt}
0
&= f\p_1\p_2h.
\end{align}
Differentiating \eqref{K1} w.r.t.~$x_1$ and \eqref{K2} w.r.t.~$x_2$	yields 
\begin{equation}
\label{doppelt}
f\p_1^2h=f\p_2^2h=0.
\end{equation}
Now we multiply \eqref{K1} by $\sigma_1$, \eqref{K2} by $\sigma_2$, 
equate the resulting equations for 
$\sigma_1 \sigma_2 h$, use \eqref{K3}, 
\eqref{sigma3}, and obtain
\begin{equation}
\label{skp}
0=\scp{b}{\sigma}.
\end{equation}
From \eqref{gemischt}, \eqref{doppelt} 
we conclude that $\na_{x'}h$ is constant on connected components of 
$V\cap\{f\neq 0\}$. Therefore, $h$ is  
an affine function on each part and continuous on the whole of $V$. 
Note that $\{f=0\}$ is a subset of the line 
$\mathcal N_{\sigma,b}:=\setlr{x'\in\reals^2}{b_3+\omega\,\sigma_1 x_2-\omega\,\sigma_2 x_1=0}$.	
Now we extend the affine function from one connected component 
of $V\cap\{f\neq 0\}$ to $\reals^2$ and call the 
resulting affine function $\tilde{h}$. 
Because of \eqref{sigma3} the plane 
$\mathcal E_{\tilde{h}}:=\setb{(x',\tilde{h}(x'))}{x'\in\reals^2}$ 
is collinear to $g_r$. 
Recalling $\scp{\sigma}{b}=0$, 
it is straightforward to check that $g_r$ is the affine kernel of $r$.
Now we use this fact together with
the collinearity of $\mathcal E_{\tilde{h}}$ and $g_r$ 
in order to prove 
$g_r\subset \mathcal E_{\tilde{h}}$.
It is sufficient to show that 
$\mathcal E_{\tilde{h}}\cap\{r=0\}$ is not void.
But in view of \eqref{K3} and \eqref{normal-equation} this is obvious.

Now let $p\in\gan$. Since $\scp{\sigma}{b}=0,$
the solutions $\gamma$ of $\dot{\gamma}=r(\gamma)$ are circles, 
contained in planes perpendicular to 
$g_r$ and with centers on $g_r$ 
(See also the computations in the proof of Lemma \ref{equivalence}.).
Hence, applying Proposition \ref{tangentialflow-lem}, 
every connected component is a subset of some 
hyper surface being axisymmetric w.r.t.~$g_r$.
\end{proof}	
	
\section{Appendix}

\begin{proof}[Proof of Proposition \ref{tangentialflow-lem}]
Clearly, it is sufficient to prove the invariance locally.
Since $\ga$ is Lipschitz, after rotation there is a neighborhood $U=V\times I$ 
of $p$ with $V\subset\reals^{N-1}$, $I\subset\reals$, 
orthonormal coordinates $(x^1,\ldots,x^N)=(x',x^N)\in V\times I$, 
a point $x'_0\in V$ and a Lipschitz continuous function 
$h:V\to I$ such that $p=(x'_0,h(x'_0))$, and for all $x\in U$ 
we have $x\in\ga$ iff $x^N= h(x')$.
By Rademacher's theorem $h$ is differentiable a.e.~with respect to the 
$(N-1)$-dimensional Lebesgue measure on $V$, and 
$\na_{x'}h\in\li(V)$.
Furthermore, the set of the $N-1$ vectors
\begin{equation*}
t_1(x'):=\big(1,0,\ldots,0,\p_1h(x')\big)^t,\,\ldots\,,
t_{N-1}(x'):=\big(0,\ldots,0,1,\p_{N-1}h(x')\big)^t
\end{equation*}
gives a basis of the tangential space of $\ga$ in the point $(x',h(x'))$ 
for almost all $x'\in V$.
Therefore, on $\ga\cap\,U$ we have two representations of the vector field $r$, 
one representation in the coordinate vectors of $x^1,\ldots,x^N$ 
holding on the whole of $U$,
$$r(x)=r_U(x)=\big(r_U^1(x),\ldots,r_U^N(x)\big)^t,$$
and the functions $r_U^i$, $i=1,\ldots,N$, are Lipschitz continuous functions on $U$.
On the other hand, for almost all $x'\in V$
$$r(x',h(x'))=r_V^1(x')t_1 (x')+\cdots+r_V^{N-1}(x')t_{N-1}(x').$$
We define $r_V:=(r_V^1,\ldots,r_V^{N-1})^t$.
Comparison yields a.e.~on $V$ and for all $i=1,\ldots,N-1$
\begin{equation}
\label{Koeffizientenfunktionen}
r_U^i(x',h(x'))=r_V^i(x').
\end{equation}
Hence, $r_V$ is Lipschitz continuous on $V$. Furthermore,
\begin{equation}
\label{N-Funktion}
r_U^N (x',h(x'))
=r_V^1(x')\p_1h(x')+\cdots+r_V^{N-1}(x')\p_{N-1}h(x')
=r_V(x')\cdot\na_{x'}h(x')
\end{equation}
holds for almost all $x'\in V$. Since $h$ is Lipschitz on $V$ and $r_U^N$ 
is Lipschitz on $U$, $r_V\cdot\na_{x'}h$ is also Lipschitz on $V$.
Now we define the flow of $r_V$: For $x'\in V$ we set $\psi(\,\cdot,x')$ 
as the solution  of the ordinary differential equation
\begin{equation}
\label{Def-ODE-V}
\dot{\psi}(t,x')
=r_V\big(\psi(t,x')\big),\quad\psi(0,x')=x'.
\end{equation}
Since $r_V$ is Lipschitz on $V$, 
we can restrict the flow such that for some $\eps>0$ and some
neighborhood $\bar{V}\subset V$ of $x_0'$ the solution $\psi$ is Lipschitz continuous on 
$(-\eps,\eps)\times\bar{V}$.
Next we lift up this flow to $\ga$ and define
$$\gamma_V(t):=\big(\psi(t,x'_0),h(\psi(t, x'_0))\big)^t.$$
By definition $\gamma_V(0)=p$ and $\gamma_V(t)\in\ga$ 
for all $t\in(-\eps,\eps)$.

In the next step we have to prove that $\gamma_V$
is also a solution of \eqref{ODE} on $(-\eps,\eps)$. 
With regard to \eqref{Koeffizientenfunktionen} 
it only remains to prove that the mapping
$t\mapsto h(\psi(t,x'_0))$ is classically differentiable with derivative  
$\p_t\big(h(\psi(t,x'_0))\big)=r^N_U\big(\psi(t,x'_0),h(\psi(t,x'_0))\big)$. 
We denote the $l$-dimensional Lebesgue measure by $\calL^l$.
For all $t\in(-\eps,\eps)$ it holds that 
$\psi(t,\,\cdot\,)$ is a bi-Lipschitz homeomorphism with
inverse Lipschitz transformation $\psi(t,\,\cdot\,)^{-1}=\psi(-t,\,\cdot\,)$. 
Therefore, if $\calL^{N-1}(\psi(t,\,\cdot\,)(\tilde{V}))=0$ 
for some set $\tilde{V}\subset\bar{V}$, 
then also $\calL^{N-1}(\tilde{V})=0$, because 
$\tilde{V}=\psi(-t,\,\cdot\,)\big(\psi(t,\,\cdot\,)(\tilde{V})\big)$.
Fix a measurable set $V_0\subset V$ such that $\calL^{N-1}(V_0)=0$ and $h$ 
is classically differentiable for every $x'\in V\setminus V_0$.
Let us define 
$$W_0:=\setlr{(t,x)\in(-\eps,\eps)\times\bar{V}}{\psi(t,x)\in V_0}.$$
Then $W_0$ is measurable and using Tonelli's and Fubini's theorems 
and the change of variable formula we obtain
\begin{align*}
\calL^N(W_0) 
&=\int_{(-\eps,\eps)\times\bar{V}}{\mathbf 1}_{W_0}
\leq c\int_{(-\eps,\eps)}\int_{V_0}1=0.
\end{align*}
Therefore, and since $\psi$ is differentiable w.r.t.~$t$ everywhere, 
we have by using \eqref{N-Funktion}
\begin{equation}
\label{derivative-formula}
\p_th(\psi(t,x'))
=\na h(\psi(t,x'))\cdot\p_t\psi(t,x')
=r^N_U\big(\psi(t, x'),h(\psi(t, x'))\big)
\end{equation}
for almost all $(t,x')\in(-\eps,\eps)\times\bar{V}$.
Consequently this formula holds in the distributional sense.
Because $h\circ\psi$ is continuous 
and its distributional derivative w.r.t.~$t$ is also continuous, 
it is also differentiable w.r.t.~$t$ in the classical sense. 
This can be seen as follows: We define 
$$v(t,x')
:=h(\psi(0,x'))
+\int_0^tr^N\big(\psi(\tau,x'),h(\psi(\tau,x'))\big)\,\mathrm{d}\tau.$$
The vector field $v$ is classically differentiable  w.r.t.~$t$ 
and $\p_tv(t, x')=r^N_U\big(\psi(t,x'),h(\psi(t,x'))\big)$ 
holds for all $(t,x')\in(-\eps,\eps)\times\bar{V}$.
Furthermore, for all $\phi\in\cic\left((-\eps,\eps)\times\bar{V}\right)$
$$\int_{(-\eps,\eps)\times\bar{V}}
(v-h\circ\psi)\p_t\phi=0.$$
This yields $h\circ\psi(t, x')=v(t, x')+w(x')$.
Since for all $x'\in\bar{V}$ we have $h\circ\psi(0,x')=v(0,x')$,
we finally conclude $w= 0$ on $\bar{V}$ 
and hence $v=h\circ\psi$.
\end{proof}

\begin{proof}[Proof of Lemma \ref{equivalence}]
For (i) we choose $\sigma_1,\sigma_2\in\rt$
such that the set $\{\sigma_1,\sigma_2,\sigma\}$ 
gives a positively oriented orthonormal basis of $\rt$.
Let $x\in\ga$ and define $d:=\dist(g,x)$.
Since $\om$ is axisymmetric w.r.t.~$g$, for all $t\in\reals$
\begin{align*}
\gamma(t)
&:=\scp{x}{\sigma}\sigma
+\big(\scp{b}{\sigma_1}+d\cos(t)\big)\sigma_1
+\big(\scp{b}{\sigma_2}+d\sin(t)\big)\sigma_2\in\ga.
\end{align*}
Therefore, $\dot{\gamma}(t)$ is a tangential vector at $\ga$ located in $x$.
On the other hand
\begin{align*}
r(x)
&=\sigma\wedge(x-b)
=\sigma_2\scp{x-b}{\sigma_1}-\sigma_1\scp{x-b}{\sigma_2}\\
&=\sigma_2\bscp{\big(\scp{b}{\sigma_1}+d\cos(t)\big)\sigma_1-b}{\sigma_1}
-\sigma_1\bscp{\big(\scp{b}{\sigma_2}+d\sin(t)\big)\sigma_2-b}{\sigma_2}\\	
&=\sigma_2d\cos \left(t\right)-\sigma_1d\sin \left(t\right)
=\dot{\gamma}\left(t\right),
\end{align*}
which yields $r\in\hocnom\cap\calR$.

No we turn to the proof of (ii). 
If $\omega=0$ then $x(t)=x_0+tb$ remains in $\ga$ for all $t$
if $x_0\in\ga$ (Proposition \ref{tangentialflow-lem}) 
and $\om$ would be unbounded. Therefore, we have $\omega\neq0$.
We choose again $\sigma_1,\sigma_2\in\rt$
such that the set $\{\sigma_1,\sigma_2,\sigma\} $ 
gives an orthonormal basis of $\rt$ with positive orientation.	
The solution of the ordinary differential equation system
\begin{align*}
\dot{s}_1
&=-\omega s_2+\scp{b}{\sigma_1},
&
\dot{s}_2
&=\omega s_1+\scp{b}{\sigma_2},
&
\dot{s}_3
&=\scp{b}{\sigma},\\
s_1(0)
&=\scp{\hat{x}}{\sigma_1},
&
s_2(0)
&=\scp{\hat{x}}{\sigma_2},
&
s_3(0)
&=\scp{\hat{x}}{\sigma}
\end{align*}
is given by
\begin{align*}
s_1(t)
&=c_1\cos(\omega t)-c_2\sin(\omega t)-\frac{1}{\omega}\scp{b}{\sigma_2},\\
s_2(t)
&=c_1\sin(\omega t)+c_2\cos(\omega t)+\frac{1}{\omega}\scp{b}{\sigma_1},\\
s_3(t)
&=\scp{\hat{x}}{\sigma}+t\scp{b}{\sigma},
\end{align*}
where $c_1$ and $c_2$ are uniquely defined 
by the initial conditions on $s_1$ and $s_2$. Then 
$$x(t):=s_1(t)\sigma_1+s_2(t)\sigma_2+s_3(t)\sigma$$
is the unique solution of 
$$\dot{x}=r(x),\quad x(0)=\hat{x}.$$
Due to Proposition \ref{tangentialflow-lem}  
and since $r\in\hocnom$, we have $x(t)\in\ga$ for all $t\in\reals$.
Because $\om$ is bounded, we conclude $\scp{b}{\sigma}=0$.
Therefore, the trajectory $t\mapsto x(t)$ 
is a circle lying in a plane perpendicular to $\sigma$ with center 
$$-\frac{1}{\omega}\scp{b}{\sigma_2}\sigma_1
+\frac{1}{\omega}\scp{b}{\sigma_1}\sigma_2
+\scp{\hat{x}}{\sigma}\sigma.$$
Consequently, $\om$ is axisymmetric w.r.t.~to $g$.
\end{proof}



\bibliographystyle{plain} 
\bibliography{paule,Litbank}

\end{document}

%% file: header-pauly.tex
\usepackage[mathscr]{eucal}
\usepackage[english]{babel}
\usepackage{a4,exscale,ifthen,amsfonts,amssymb,amsmath,amscd,graphicx,color}
\usepackage{nicefrac,tikz,fancyhdr,caption,array,multirow,multicol,booktabs,algorithm,algorithmic}
\usepackage[all]{xy}

\ifthenelse{\equal{\mylabelonoff}{on}}
{\newcommand{\mylabel}[1]{\label{#1}\fbox{{\sf #1}}}}
{\newcommand{\mylabel}[1]{\label{#1}}}
\ifthenelse{\equal{\allowdisbrkyesno}{yes}}
{\allowdisplaybreaks}
{}

\ifthenelse{\equal{\pagesizeextendednormal}{extended}}
{\setlength{\textwidth}{16cm}
\setlength{\textheight}{22cm}
\setlength{\oddsidemargin}{-0.2cm}
\setlength{\evensidemargin}{-0.2cm}}
{}
\ifthenelse{\equal{\numberingequationsectionyesno}{yes}}
{\numberwithin{equation}{section}}
{}

\newcommand{\ovl}[1]{\overline{#1}}

\newcommand{\ck}{c_{\mathsf{k}}}

\newcommand{\set}[2]{\{#1\,:\,#2\}}
\newcommand{\setb}[2]{\big\{#1\,:\,#2\big\}}
\newcommand{\setlr}[2]{\left\{#1\,:\,#2\right\}}

\ifthenelse{\equal{\numberingtheoremsectionyesno}{yes}}
{\newtheorem{lem}{Lemma}[section]}
{\newtheorem{lem}{Lemma}}
\newtheorem{defi}[lem]{Definition}
\newtheorem{theo}[lem]{Theorem}
\newtheorem{cor}[lem]{Corollary}
\newtheorem{rem}[lem]{Remark}
\newtheorem{pro}[lem]{Proposition}

\newcommand{\om}{\Omega}

\newcommand{\ga}{\Gamma}
\newcommand{\gat}{\ga_{\mathsf{t}}}
\newcommand{\gan}{\ga_{\mathsf{n}}}

\newcommand{\eps}{\epsilon}

\newcommand{\calR}{\mathcal{R}}
\newcommand{\calN}{\mathcal{N}}

\newcommand{\reals}{\mathbb{R}}

\newcommand{\rt}{\reals^{3}}
\newcommand{\rN}{\reals^{N}}

\newcommand{\ot}{\leftarrow}

\DeclareMathOperator{\id}{id}
\DeclareMathOperator{\sym}{sym}
\DeclareMathOperator{\skw}{skw}

\DeclareMathOperator{\dist}{dist}

\newcommand{\tcomp}[1]{#1_{\mathsf{t}}}
\newcommand{\ncomp}[1]{#1_{\mathsf{n}}}

\newcommand{\p}{\partial}
\newcommand{\na}{\nabla}

\DeclareMathOperator{\divergence}{div}
\renewcommand{\div}{\divergence}

\newcommand{\csymbol}{\mathsf{C}}
\newcommand{\cgen}[3]{\overset{#1}{\csymbol}{}^{#2}_{#3}}

\newcommand{\cic}{\cgen{\circ}{\infty}{}}
\newcommand{\cict}{\cgen{\circ}{\infty}{\mathsf{t}}}
\newcommand{\cicn}{\cgen{\circ}{\infty}{\mathsf{n}}}
\newcommand{\cictn}{\cgen{\circ}{\infty}{\mathsf{t,n}}}

\newcommand{\co}{\cgen{}{1}{}}

\newcommand{\coo}{\cgen{}{1,1}{}}

\newcommand{\ct}{\cgen{}{2}{}}

\newcommand{\cicrN}{\cic(\rN)}

\newcommand{\cicom}{\cic(\om)}
\newcommand{\cictom}{\cict(\om)}
\newcommand{\cicnom}{\cicn(\om)}
\newcommand{\cictnom}{\cictn(\om)}

\newcommand{\lsymbol}{\mathsf{L}}
\newcommand{\lgen}[3]{\overset{#1}{\lsymbol}{}^{#2}_{#3}}

\newcommand{\lt}{\lgen{}{2}{}}

\newcommand{\li}{\lgen{}{\infty}{}}

\newcommand{\ltz}{\lgen{}{2}{0}}

\newcommand{\ltom}{\lt(\om)}

\newcommand{\ltzom}{\ltz(\om)}

\newcommand{\hsymbol}{\mathsf{H}}

\newcommand{\hgen}[3]{\overset{#1}{\hsymbol}{}^{#2}_{#3}}

\newcommand{\ho}{\hgen{}{1}{}}

\newcommand{\hoc}{\hgen{\circ}{1}{}}

\newcommand{\hoct}{\hgen{\circ}{1}{\mathsf{t}}}
\newcommand{\hocn}{\hgen{\circ}{1}{\mathsf{n}}}
\newcommand{\hoctn}{\hgen{\circ}{1}{\mathsf{t,n}}}

\newcommand{\hoom}{\ho(\om)}

\newcommand{\hocom}{\hoc(\om)}

\newcommand{\hoctom}{\hoct(\om)}
\newcommand{\hocnom}{\hocn(\om)}
\newcommand{\hoctnom}{\hoctn(\om)}

\newcommand{\hmo}{\hgen{}{-1}{}}
\newcommand{\hmoom}{\hmo(\om)}

\newcommand{\norm}[1]{|#1|}

\newcommand{\normlt}[1]{\norm{#1}_{\lt}}

\newcommand{\normltom}[1]{\norm{#1}_{\ltom}}

\newcommand{\normhmoom}[1]{\norm{#1}_{\hmoom}}

\newcommand{\scp}[2]{\langle#1,#2\rangle}

\newcommand{\bscp}[2]{\big\langle#1,#2\big\rangle}

\newcommand{\scpltom}[2]{\scp{#1}{#2}_{\ltom}}

\newcommand{\preprintudemath}[5]{
\thispagestyle{empty}
\Large
\begin{center}SCHRIFTENREIHE DER FAKULT\"AT F\"UR MATHEMATIK\end{center}
\vspace*{5mm}
\begin{center}#1\end{center}
\vspace*{5mm}
\begin{center}by\end{center}
\begin{center}#2\end{center}
\vspace*{5mm}
\begin{center}#3\hspace{80mm}#4\end{center}
\newpage
\thispagestyle{empty}
\vspace*{210mm}
Received: #5
\newpage
\addtocounter{page}{-2}
\normalsize}